\documentclass[]{article}
\usepackage{amsmath,amsthm,amssymb}
\title{Tait coloring and a moduli space}
\author{Zipei Zhuang}

\usepackage{pifont}
\usepackage{indentfirst}
\usepackage{setspace}
\usepackage{mathrsfs}
\usepackage[backref]{hyperref}

\usepackage{subfigure}
\usepackage{graphicx}
\usepackage{import}
\usepackage{xifthen}
\usepackage{pdfpages}
\usepackage{transparent}
\newtheorem{definition}{Definition}%
\newtheorem{theorem}{Theorem}%
\newtheorem{lemma}{Lemma}%

\newtheorem{question}{Question}%
\newtheorem{conjecture}{Conjecture}
\begin{document}
	
	\maketitle
	
	\begin{abstract}
	We associate a moduli space $\mathcal{M}(G)$ to a planar trivalent graph $G$. We proved several decomposition properties of $\mathcal{M}(G)$, which implies that the Euler characteristic of $\mathcal{M}(G)$ equals to the number of Tait colorings of $G$ when $G$ is bipartite. Then we interpret  $\mathcal{M}(G)$ as a representation space of the fundamental group of $G$ to $SU(3)$.
	\end{abstract}
	
	\section{Introduction}
	
	Let $G$ be a trivalent graph. A Tait coloring of $G$ is a function from the edges of $G$ to a 3-element set of "colors" \{1,2,3\} such that edges of 3 different colors are incident at each vertex. The four-color theorem is equivalent to the existence of Tait colorings for bridgeless planar trivalent graphs.
	
	There is a similar concept in knot theory. A \textbf{web} is an oriented trivalent planar graph, such that at each vertex the edges are either all oriented in or oriented out. In particular, such graphs are always bipartite. The invariant $P_3(\Gamma)$ (\cite{MR1403861}\cite{MR1659228}) of a web $\Gamma$, which is used to construct the $sl_3$ knot polynomial invariant, is characterized by the properties illustrated in Fig. \ref{graph0}  ,  where $[n]=\frac{q^n -q^{-n}}{q-q^{-1}}$, since a closed bipartite graph contains either a loop, a digon face or a square face.
	
	Let $q=1$, we see that relations in Fig.\ref{graph0}  ensure that $P_3(\Gamma)(1)$ is the number $|Tait(\Gamma)|$ of Tait colorings of $\Gamma$. This coincidence raised people's interest on modifying concepts and methods devoloped for studying $P_n(\Gamma)$ in knot theory to study $|Tait(G)|$ for a general trivalent graph $G$.

	In \cite{MR2100691}, Khovanov defined a link homology theory using webs(oriented closed trivalent graphs) and foams(cobordisms between webs), whose Euler characteristic is the $sl_3$ link polynomial. Kronheimer and Mrowka(\cite{MR3880205}) defined an instanton homology $J^{\#}(G)$ for unoriented trivalent spatial graphs $G$, whose dimension is conjectured to be the number of Tait colorings of $G$. The non-vanishing theorem they proved would then, lead to a new proof of the four-color theorem. They also proposed a combinatorial counterpart to $J^{\#}$ for planar unoriented webs, by imitating the construction of Khovanov's $sl_3$ homology in \cite{MR2100691}.
	
	Motivated by the combinatorial construction in \cite{MR3880205}, Khovanov and Robert developed a homology theory $<G>$ for planar unoriented trivalent graphs $G$ using an unoriented version of the Robert-Wagner foam evaluation, whose rank is conjecturally related to $|Tait(G)|$. 
	
	Motivated by the perspective of a potential relationship between the $SU(N)$ knot instanton homology theory \cite{MR2860345} \cite{MR3966740} and Khovanov-Rozansky's $sl(N)$ knot homology, the authors of \cite{MR3190356} defined a moduli space $\mathcal{M}(\Gamma)$ for any web $\Gamma$, and showed that the Euler characteristic of $\mathcal{M}(\Gamma)$ equals the evaluation of $P_{N}(\Gamma)$ at 1.

	In this paper, we use a modified construction in \cite{MR3190356} to aassociate a moduli space $\mathcal{M}(G)$ to a planar trivalent graph $G$. 
	In Section 2, we defined and proved several decomposition properties of $\mathcal{M}(G)$, which shows that the Euler characteristic $\chi(\mathcal{M}(G))$ behaves the same as the number of Tait colorings $|Tait(G)|$ of $G$ under some decompositions. In particular, this shows that 
	\begin{equation}
		\chi(\mathcal{M}(G))=|Tait(G)|
	\end{equation}
for any bipartite graph $G$.

In Section 3, we showed that $\mathcal{M}(G)$ is homeomorphic to the representation space $R_\Phi(\pi_1(G) ; SU(3))$, which is similar to the representation spaces appearing in Kronheimer and Mrowka's instanton homology  \cite{MR3880205}. \cite{MR3966740}.

\begin{figure}[t]
	\def\svgwidth{\columnwidth}
	%% Creator: Inkscape 1.0 (4035a4fb49, 2020-05-01), www.inkscape.org
%% PDF/EPS/PS + LaTeX output extension by Johan Engelen, 2010
%% Accompanies image file 'graph0.pdf' (pdf, eps, ps)
%%
%% To include the image in your LaTeX document, write
%%   \input{<filename>.pdf_tex}
%%  instead of
%%   \includegraphics{<filename>.pdf}
%% To scale the image, write
%%   \def\svgwidth{<desired width>}
%%   \input{<filename>.pdf_tex}
%%  instead of
%%   \includegraphics[width=<desired width>]{<filename>.pdf}
%%
%% Images with a different path to the parent latex file can
%% be accessed with the `import' package (which may need to be
%% installed) using
%%   \usepackage{import}
%% in the preamble, and then including the image with
%%   \import{<path to file>}{<filename>.pdf_tex}
%% Alternatively, one can specify
%%   \graphicspath{{<path to file>/}}
%% 
%% For more information, please see info/svg-inkscape on CTAN:
%%   http://tug.ctan.org/tex-archive/info/svg-inkscape
%%
\begingroup%
  \makeatletter%
  \providecommand\color[2][]{%
    \errmessage{(Inkscape) Color is used for the text in Inkscape, but the package 'color.sty' is not loaded}%
    \renewcommand\color[2][]{}%
  }%
  \providecommand\transparent[1]{%
    \errmessage{(Inkscape) Transparency is used (non-zero) for the text in Inkscape, but the package 'transparent.sty' is not loaded}%
    \renewcommand\transparent[1]{}%
  }%
  \providecommand\rotatebox[2]{#2}%
  \newcommand*\fsize{\dimexpr\f@size pt\relax}%
  \newcommand*\lineheight[1]{\fontsize{\fsize}{#1\fsize}\selectfont}%
  \ifx\svgwidth\undefined%
    \setlength{\unitlength}{419.52755906bp}%
    \ifx\svgscale\undefined%
      \relax%
    \else%
      \setlength{\unitlength}{\unitlength * \real{\svgscale}}%
    \fi%
  \else%
    \setlength{\unitlength}{\svgwidth}%
  \fi%
  \global\let\svgwidth\undefined%
  \global\let\svgscale\undefined%
  \makeatother%
  \begin{picture}(1,0.70945946)%
    \lineheight{1}%
    \setlength\tabcolsep{0pt}%
    \put(0,0){\includegraphics[width=\unitlength,page=1]{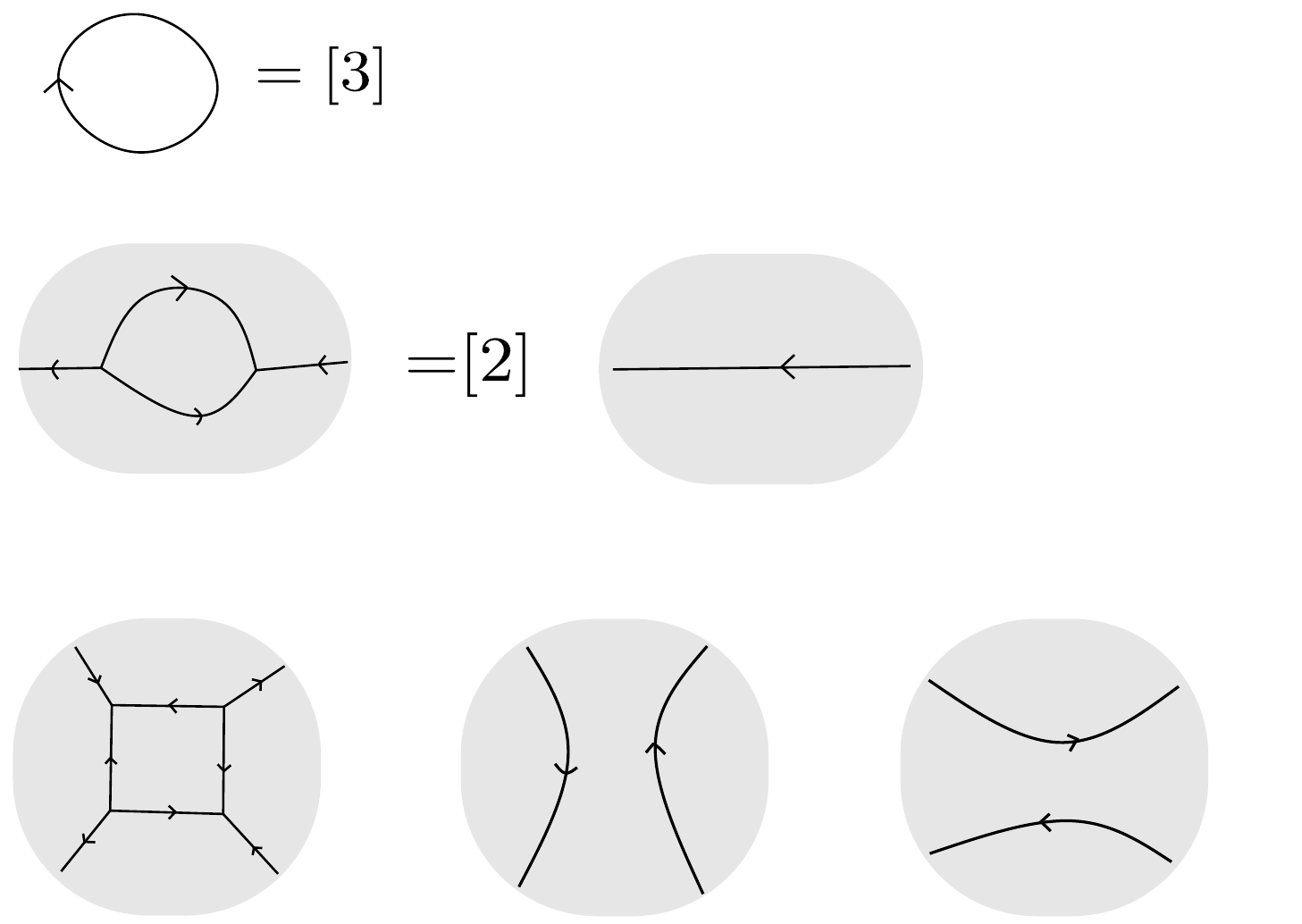}}%
    \put(0.27900773,0.10087951){\color[rgb]{0,0,0}\makebox(0,0)[lt]{\lineheight{1.25}\smash{\begin{tabular}[t]{l}=\end{tabular}}}}%
    \put(0.61760928,0.09636482){\color[rgb]{0,0,0}\makebox(0,0)[lt]{\lineheight{1.25}\smash{\begin{tabular}[t]{l}+\end{tabular}}}}%
  \end{picture}%
\endgroup%

	\caption{Graph relations that determine $P_3(\Gamma)$} 
	\label{graph0}
\end{figure}

\section{Definitions and Relations}

Let $G$ be a planar trivalent graph, $ m:=|E(G)|$ be the number of edges of $G$.

For each edge $e \in E(G)$, we decorate it by a point $D(e) \in \mathbb{CP}^2$, i.e. a line in $\mathbb{C}^3$ that passes through the origin point. We call this decoration \textbf{admissible}, if at any vertex $v$, the decorations $D(e_i)$ associated to the three edges $e_1,e_2,e_3$ adjacent to $v$ are orthogonal pairwise in $\mathbb{C}^3$.

\begin{definition}
	The set of all admissible decorations of $G$ forms a moduli space which we denote by $\mathcal{M}(G)$. More precisely, a decoration can be viewed as a point in $\oplus_{e \in E(G)} \mathbb{P}^2$, and the structure of $\mathcal{M}(G)$ is induced as a subvariety of $\oplus_{e \in E(G)} \mathbb{P}^2$.
\end{definition}

The main result of this section is:
\begin{theorem} \label{thm1}
	For a bipartite trivalent graph $G$ we have 
	\begin{equation}
		\chi(\mathcal{M}(G))=|Tait(G)|
	\end{equation}
i.e. the Euler characteristic of $\mathcal{M}(G)$ equals to the number of Tait colorings of $G$.
\end{theorem}

\begin{question}
	Is Theorem \ref{thm1}   true for any planar trivalent graph?
\end{question}

Motivated by the conjecture in \cite{MR3190356} we propose:

\begin{conjecture}
	For any planar trivalent graph $G$ and for $i$ odd, we have $H_i(\mathcal{M}(G))=0$.
\end{conjecture}

The conjecture toghther with Theorem     would together imply that $ dim H(\mathcal{M}(G))= |Tait(G)|$ for any bipartite  $G$.

\begin{lemma} \label{l1}
	1. Let $U$ be the graph with a single circle. Then $\mathcal{M}(U)$ $\cong$ $\mathbb{P}^2$, therefore
	  $\chi(\mathcal{M}(U))=|Tait(U)| =3$.
	
	2. Let $G_1, G_2$ be two trivalent graph, $G_1 \cup G_2$ their disjoint union, then
	\begin{equation}
		\mathcal{M}(G_1 \cup G_2) = \mathcal{M}(G_1) \times \mathcal{M}(G_2)
	\end{equation}
In particular,
\begin{equation}
	\chi(\mathcal{M}(G_1 \cup G_2)) = \chi(\mathcal{M}(G_1)) \cdot \chi(\mathcal{M}(G_2))
\end{equation}
\end{lemma}
\begin{proof}
	This is immediate from the definition.
\end{proof}

\begin{figure}[t]
	\def\svgwidth{\columnwidth}
	%% Creator: Inkscape 1.0 (4035a4fb49, 2020-05-01), www.inkscape.org
%% PDF/EPS/PS + LaTeX output extension by Johan Engelen, 2010
%% Accompanies image file 'graph1.pdf' (pdf, eps, ps)
%%
%% To include the image in your LaTeX document, write
%%   \input{<filename>.pdf_tex}
%%  instead of
%%   \includegraphics{<filename>.pdf}
%% To scale the image, write
%%   \def\svgwidth{<desired width>}
%%   \input{<filename>.pdf_tex}
%%  instead of
%%   \includegraphics[width=<desired width>]{<filename>.pdf}
%%
%% Images with a different path to the parent latex file can
%% be accessed with the `import' package (which may need to be
%% installed) using
%%   \usepackage{import}
%% in the preamble, and then including the image with
%%   \import{<path to file>}{<filename>.pdf_tex}
%% Alternatively, one can specify
%%   \graphicspath{{<path to file>/}}
%% 
%% For more information, please see info/svg-inkscape on CTAN:
%%   http://tug.ctan.org/tex-archive/info/svg-inkscape
%%
\begingroup%
  \makeatletter%
  \providecommand\color[2][]{%
    \errmessage{(Inkscape) Color is used for the text in Inkscape, but the package 'color.sty' is not loaded}%
    \renewcommand\color[2][]{}%
  }%
  \providecommand\transparent[1]{%
    \errmessage{(Inkscape) Transparency is used (non-zero) for the text in Inkscape, but the package 'transparent.sty' is not loaded}%
    \renewcommand\transparent[1]{}%
  }%
  \providecommand\rotatebox[2]{#2}%
  \newcommand*\fsize{\dimexpr\f@size pt\relax}%
  \newcommand*\lineheight[1]{\fontsize{\fsize}{#1\fsize}\selectfont}%
  \ifx\svgwidth\undefined%
    \setlength{\unitlength}{419.52755906bp}%
    \ifx\svgscale\undefined%
      \relax%
    \else%
      \setlength{\unitlength}{\unitlength * \real{\svgscale}}%
    \fi%
  \else%
    \setlength{\unitlength}{\svgwidth}%
  \fi%
  \global\let\svgwidth\undefined%
  \global\let\svgscale\undefined%
  \makeatother%
  \begin{picture}(1,0.70945946)%
    \lineheight{1}%
    \setlength\tabcolsep{0pt}%
    \put(0,0){\includegraphics[width=\unitlength,page=1]{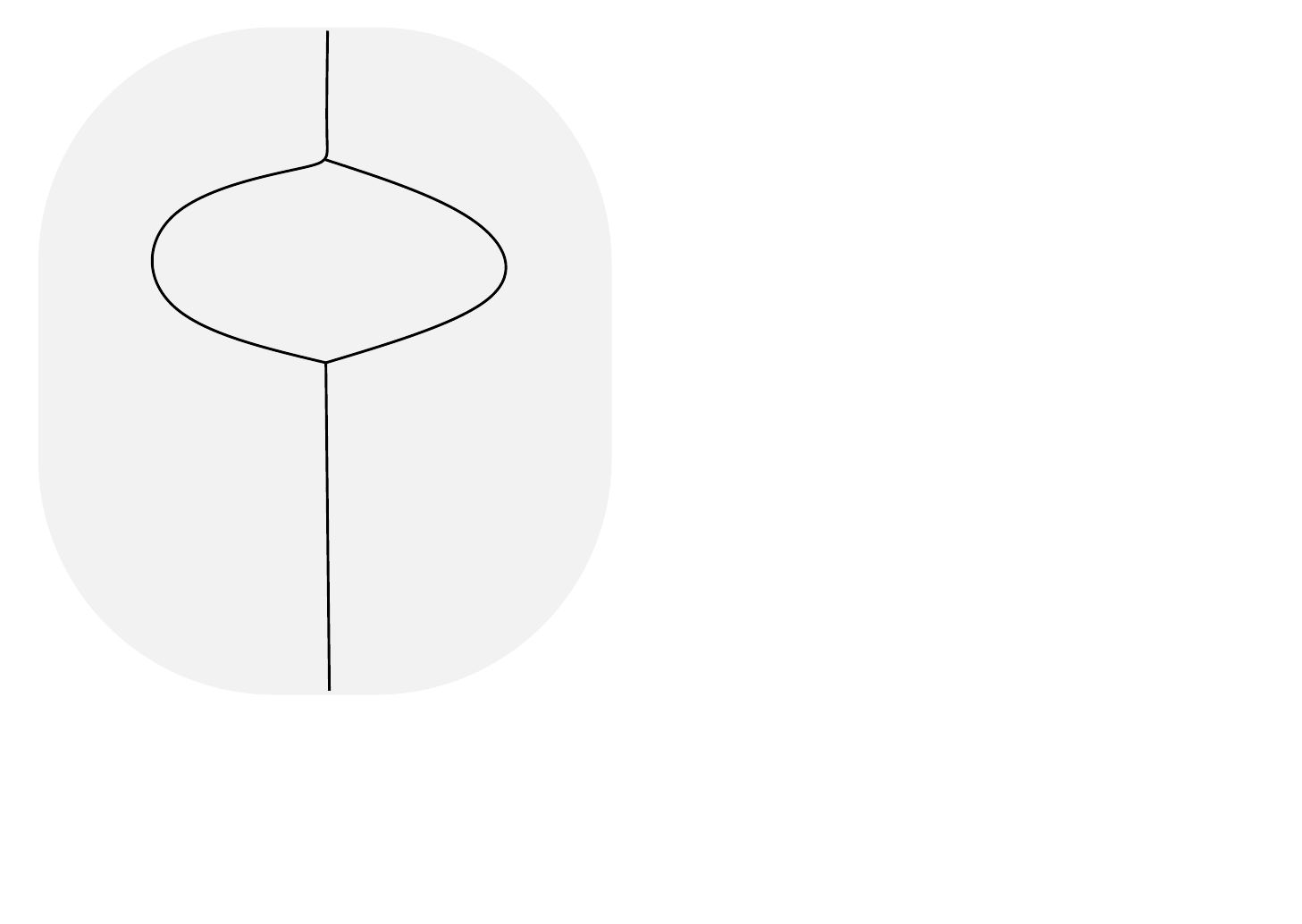}}%
    \put(0.27712964,0.62676942){\color[rgb]{0,0,0}\makebox(0,0)[lt]{\lineheight{1.25}\smash{\begin{tabular}[t]{l}a\\\end{tabular}}}}%
    \put(0.41248875,0.49764326){\color[rgb]{0,0,0}\makebox(0,0)[lt]{\lineheight{1.25}\smash{\begin{tabular}[t]{l}d\\\end{tabular}}}}%
    \put(0.06590582,0.49279554){\color[rgb]{0,0,0}\makebox(0,0)[lt]{\lineheight{1.25}\smash{\begin{tabular}[t]{l}c\end{tabular}}}}%
    \put(0.28894449,0.29738717){\color[rgb]{0,0,0}\makebox(0,0)[lt]{\lineheight{1.25}\smash{\begin{tabular}[t]{l}b\end{tabular}}}}%
    \put(0,0){\includegraphics[width=\unitlength,page=2]{graph1.pdf}}%
    \put(0.22474261,0.08630948){\color[rgb]{0,0,0}\makebox(0,0)[lt]{\lineheight{1.25}\smash{\begin{tabular}[t]{l}G\end{tabular}}}}%
    \put(0.72402872,0.07737088){\color[rgb]{0,0,0}\makebox(0,0)[lt]{\lineheight{1.25}\smash{\begin{tabular}[t]{l}G'\end{tabular}}}}%
  \end{picture}%
\endgroup%

	\caption{The bigon move} 
	\label{graph1}
\end{figure}

\begin{lemma}(The bigon relation) \label{l2}
	
	Let $G$ be a trivalent graph containing a bigon, and $G'$ the graph obtained from $G$ by collapsing the bigon(see Figure.\ref{graph1}).Then $\mathcal{M}(G)$ is a $\mathbb{CP}^1$-bundle over $\mathcal{M}(G')$. In particular, 
	\begin{equation}
		\chi(\mathcal{M}(G))= \chi(\mathbb{CP}^1) \cdot \chi(\mathcal{M}(G'))=3\chi(\mathcal{M}(G'))
	\end{equation}
\end{lemma}
\begin{proof}
	Let $\mathcal{D}$ be a decoration of $G$ as illustrated in Figure.\ref{graph1}. We have $a=b$ since they are both orthogonal to two different lines $c,d$. For fixed $a$, we can choose $c$ to be any line orthogonal to $a$, and then $d$ is the unique line orthogonal to both $a$ and $c$. Therefore for each admissible decoration of $G'$, there corresonds to a $\mathbb{CP}^1$-set of admissible decorations of $G$.
\end{proof}

\begin{figure}[t]
	\def\svgwidth{\columnwidth}
	%% Creator: Inkscape 1.0 (4035a4fb49, 2020-05-01), www.inkscape.org
%% PDF/EPS/PS + LaTeX output extension by Johan Engelen, 2010
%% Accompanies image file 'graph2.pdf' (pdf, eps, ps)
%%
%% To include the image in your LaTeX document, write
%%   \input{<filename>.pdf_tex}
%%  instead of
%%   \includegraphics{<filename>.pdf}
%% To scale the image, write
%%   \def\svgwidth{<desired width>}
%%   \input{<filename>.pdf_tex}
%%  instead of
%%   \includegraphics[width=<desired width>]{<filename>.pdf}
%%
%% Images with a different path to the parent latex file can
%% be accessed with the `import' package (which may need to be
%% installed) using
%%   \usepackage{import}
%% in the preamble, and then including the image with
%%   \import{<path to file>}{<filename>.pdf_tex}
%% Alternatively, one can specify
%%   \graphicspath{{<path to file>/}}
%% 
%% For more information, please see info/svg-inkscape on CTAN:
%%   http://tug.ctan.org/tex-archive/info/svg-inkscape
%%
\begingroup%
  \makeatletter%
  \providecommand\color[2][]{%
    \errmessage{(Inkscape) Color is used for the text in Inkscape, but the package 'color.sty' is not loaded}%
    \renewcommand\color[2][]{}%
  }%
  \providecommand\transparent[1]{%
    \errmessage{(Inkscape) Transparency is used (non-zero) for the text in Inkscape, but the package 'transparent.sty' is not loaded}%
    \renewcommand\transparent[1]{}%
  }%
  \providecommand\rotatebox[2]{#2}%
  \newcommand*\fsize{\dimexpr\f@size pt\relax}%
  \newcommand*\lineheight[1]{\fontsize{\fsize}{#1\fsize}\selectfont}%
  \ifx\svgwidth\undefined%
    \setlength{\unitlength}{419.52755906bp}%
    \ifx\svgscale\undefined%
      \relax%
    \else%
      \setlength{\unitlength}{\unitlength * \real{\svgscale}}%
    \fi%
  \else%
    \setlength{\unitlength}{\svgwidth}%
  \fi%
  \global\let\svgwidth\undefined%
  \global\let\svgscale\undefined%
  \makeatother%
  \begin{picture}(1,0.70945946)%
    \lineheight{1}%
    \setlength\tabcolsep{0pt}%
    \put(0,0){\includegraphics[width=\unitlength,page=1]{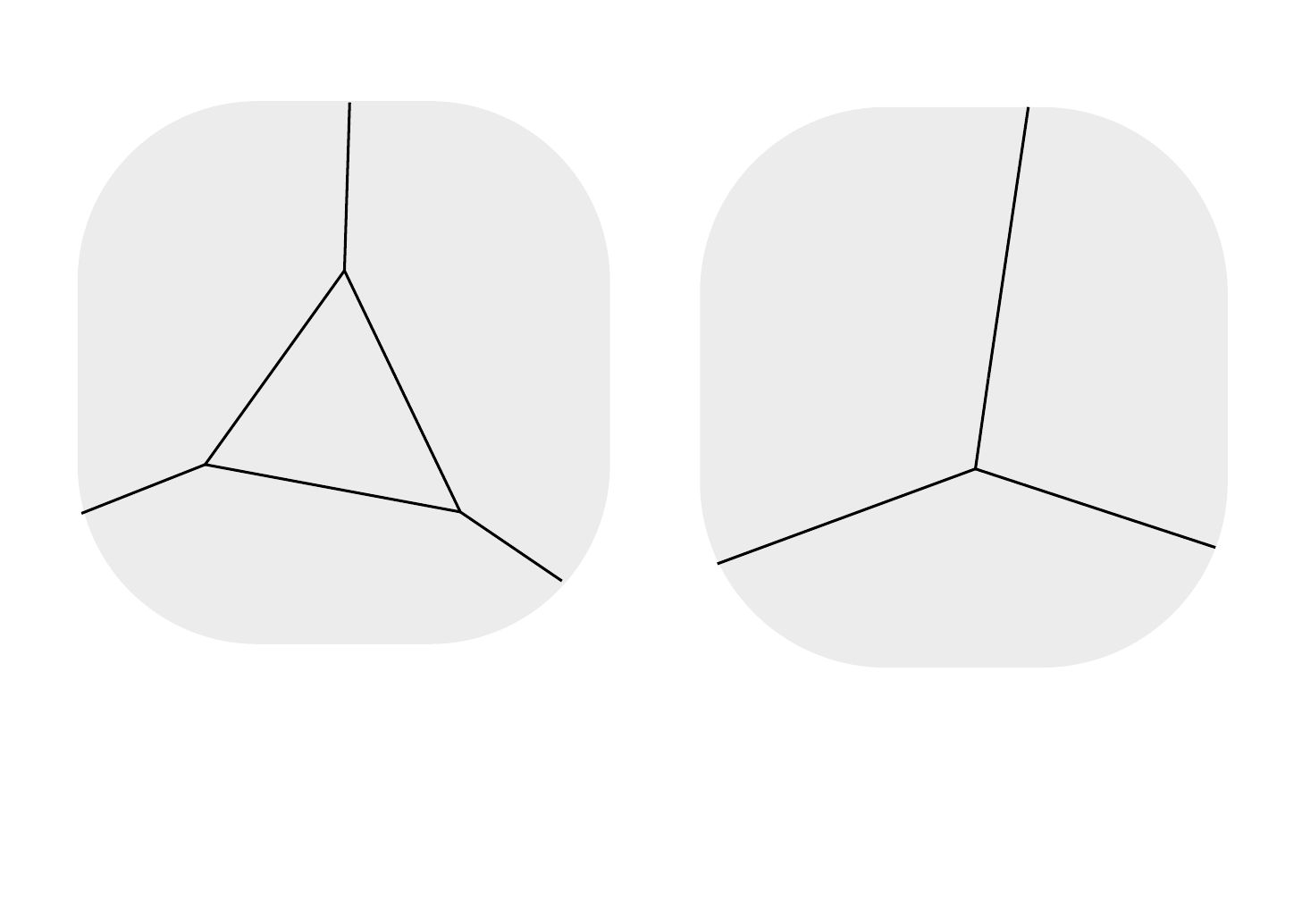}}%
    \put(0.28347374,0.55284044){\color[rgb]{0,0,0}\makebox(0,0)[lt]{\lineheight{1.25}\smash{\begin{tabular}[t]{l}a\end{tabular}}}}%
    \put(0.17057985,0.43580307){\color[rgb]{0,0,0}\makebox(0,0)[lt]{\lineheight{1.25}\smash{\begin{tabular}[t]{l}d\end{tabular}}}}%
    \put(0.32139067,0.42053229){\color[rgb]{0,0,0}\makebox(0,0)[lt]{\lineheight{1.25}\smash{\begin{tabular}[t]{l}e\end{tabular}}}}%
    \put(0.24055453,0.28527942){\color[rgb]{0,0,0}\makebox(0,0)[lt]{\lineheight{1.25}\smash{\begin{tabular}[t]{l}f\end{tabular}}}}%
    \put(0.09517235,0.34740804){\color[rgb]{0,0,0}\makebox(0,0)[lt]{\lineheight{1.25}\smash{\begin{tabular}[t]{l}b\end{tabular}}}}%
    \put(0.40012611,0.29697014){\color[rgb]{0,0,0}\makebox(0,0)[lt]{\lineheight{1.25}\smash{\begin{tabular}[t]{l}c\end{tabular}}}}%
    \put(0.19774333,0.0755973){\color[rgb]{0,0,0}\makebox(0,0)[lt]{\lineheight{1.25}\smash{\begin{tabular}[t]{l}G\end{tabular}}}}%
    \put(0.73408833,0.0765002){\color[rgb]{0,0,0}\makebox(0,0)[lt]{\lineheight{1.25}\smash{\begin{tabular}[t]{l}G'\end{tabular}}}}%
  \end{picture}%
\endgroup%

	\caption{The triangle move} 
	\label{graph2}
\end{figure}
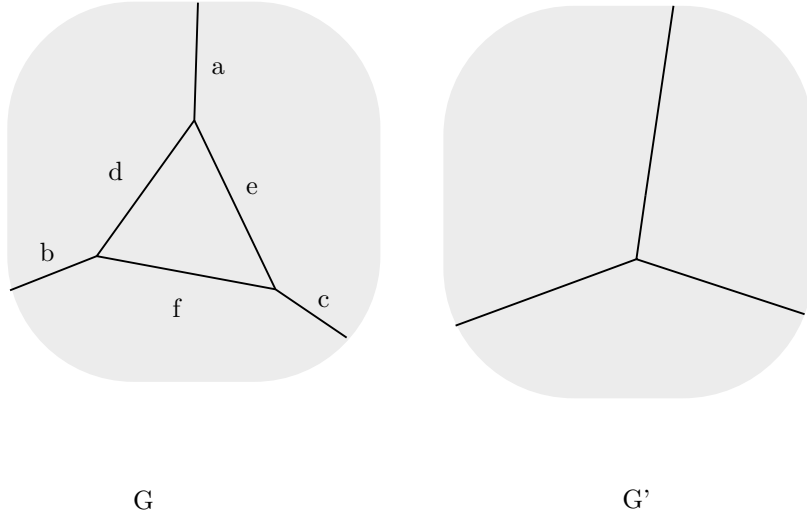

\begin{lemma}(The triangle relation) \label{l3}
	Let $G$ be a trivalent graph containing 3 edges which form a triangle. Let $G'$ be obtained from $G$ by collapsing the triangle to a single vertex(see Figure. \ref{graph2} ). Then $\mathcal{M}(G) \cong \mathcal{M}(G')$. 
\end{lemma}

\begin{proof}
	Let $\mathcal{D}$ be an admissible decoration of $G$, as illustrated in Figure.\ref{graph2}. Since $e \neq f$, we have $a \neq b$, therefore $a,b,c$ are 3 different edges, and $d,e,f$ are unniquely determined by $a,b,c$. It follows that admissible decorarions of $G$ and $G'$ are 1-1 identified.
\end{proof}

\begin{figure}[t]
	\def\svgwidth{\columnwidth}
	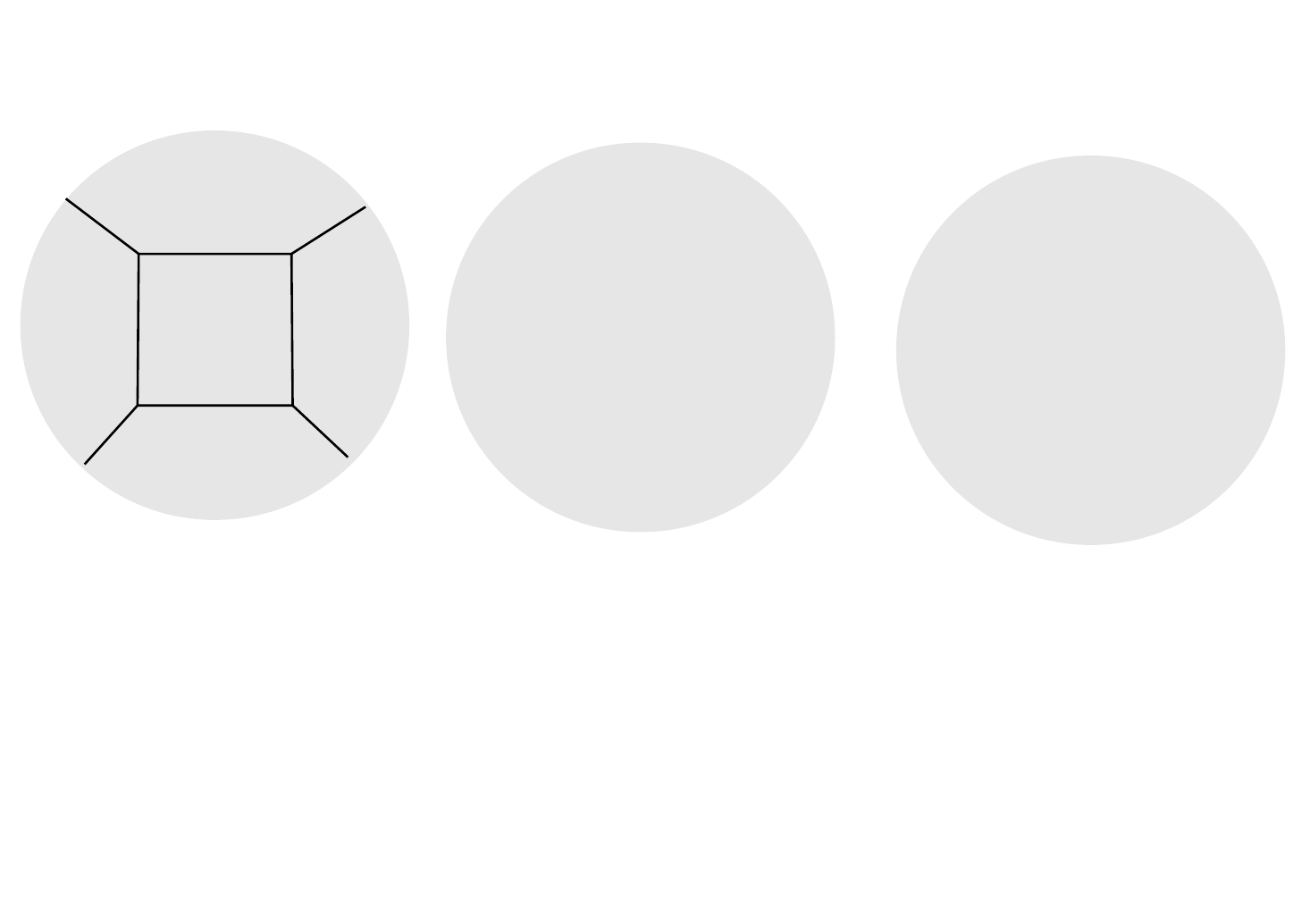
	\caption{The square move} 
	\label{graph3}
\end{figure}

\begin{lemma}(The square relation) \label{l4}
	Let $G$ be a trivalent graph containing a square: 4 edges connecting 4 vertices in a standard disk. Let $G'$ and $G''$ be obtained from $G$ as shown in Figure. \ref{graph3}
	
	Then
	\begin{equation}
		\chi(\mathcal{M}(G))=	\chi(\mathcal{M}(G'))+	\chi(\mathcal{M}(G''))
	\end{equation}
\end{lemma}

\begin{proof}
	Let $\mathcal{D}$ be an admissible decoration of $G$ as illustrated in the left of  Figure.\ref{graph3}    . If $x \neq z$, then $w=y$, since they are orthogonal to 2 different lines. In this case we have $a=d,b=c$. Similarly if $w \neq y$, then $x=y, a=b, c=d$.
	
	In the case when $a=b\neq c=d$, $x$ and $z$ is uniquely determined, and $w,y$ is determined afterward. So there is a unique choice for the interior decorations. The same result is true when $a=d \neq b=c$. When $a=b=c=d$, we have a $\mathbb{CP}^1$ of choice of, e.g. $w$, and the decorarions $x,y,z$ is determined thereafter.
	
	Consider the algebraic maps from $\mathcal{M}(G), \mathcal{M}(G'), \mathcal{M}(G)''$ to $(\mathbb{CP}^2)^4$, defined by evaluating at $a,b,c,d$. Write $V, V_1, V_2$ for the preimages of the subvariety $\triangle$ of $(\mathbb{CP}^2)^4$ given by $a=b=c=d$. Since $\mathcal{M}(G),\mathcal{M}(G'),\mathcal{M}(G'')$ are compact real varieties, we can choose a small enough open set $U \subset (\mathbb{CP}^2)^4$ containing $\triangle$, so that its preimages $\widetilde{V}, \widetilde{V_1},\widetilde{V_2}$ in $\mathcal{M}(G),\mathcal{M}(G'),\mathcal{M}(G'')$ are respectively homotopy equivalent to $V, V_1, V_2$.

	We have shown that $V$ is a $\mathbb{CP}^1$- bundle over $V_1=V_2$, therefore 
	\begin{equation}
		\chi (\widetilde{V})=\chi(V) =2 \chi(V_1)= \chi (V_1)+ \chi(V_2)
	\end{equation}
Since $\mathcal{M}(G)\backslash V, \widetilde{V}$ form an open cover of $\mathcal{M}(G)$,using the Mayer-Vietoris sequence we have:
\begin{equation}
	\chi(\mathcal{M}(G))=\chi(\mathcal{M}(G)\backslash V) + \chi (\widetilde{V}) - \chi (\widetilde{V} \backslash V)
	\end{equation}

	Notice that $\mathcal{M}(G) \backslash V $ is the disjoint union of $\mathcal{M}(G_1) \backslash V_1$ and $\mathcal{M}(G_2) \backslash V_2$, and $\widetilde{V} \backslash V$ is the disjoint union of $\widetilde{V_1} \backslash V_1$ and $\widetilde{V_2} \backslash V_2$, we have
	
	\begin{equation}
		\begin{aligned}
				\chi(\mathcal{M}(G))= & \chi({\mathcal{M}(G_1) \backslash V_1}) +\chi(V_1) -\chi({\widetilde{V_1} \backslash V_1}) \\
			& \quad +\chi({\mathcal{M}(G_2) \backslash V_2}) +\chi({V_2})-\chi({\widetilde{V_2} \backslash V_2})
			\\
			=& \chi(\mathcal{M}(G'))+	\chi(\mathcal{M}(G''))
		\end{aligned}
	\end{equation}

\end{proof}

\textbf{Proof of Theorem \ref{thm1}   }
   We can proceed by induction on the number $n$ of edges of $G$. When $n=1$, $G$ is a circle, $\mathcal{M}(G)=\mathbb{CP}^2$, $\chi (\mathcal{M}(G))=3= |Tait(G)|$.  Since a closed bipartite graph contains either a loop, a digon face or a square face, using Lemma  \ref{l1}-\ref{l4}    we can reduce $G$ to graphs with fewer number of edges, for which the inductive hypothesis is true. Notice that properties of $\chi (\mathcal{M}(G))$ from Lemma \ref{l1}- \ref{l4}  are also true for $|Tait(G)|$, the inductive process works.

\section{Relation to representation spaces}

Let
\begin{equation}
	\phi= \begin{pmatrix}
		1 & 0 & 0 \\ 0 & -1 & 0 \\ 0 & 0 & -1
	\end{pmatrix}
\in SU(3)
\end{equation}

Any matrix $A \in SU(3)$ of order 2 is diagonalizable and is conjugate to $\Phi$.

Let $G \subset \mathbb{R}^2 \subset \mathbb{R}^3 $ be a trivalent planar graph.
Choose an orientation on each edge of $G$ so that it induce a meridian around the edge. The meridians generate the fundamental group $\pi_1(G):= \pi_1(\mathbb{R}^3 \backslash G)$ of $G$.

Define $R_{\Phi}(\pi_1(G) ; SU(3))$ to be the subspace of $Hom(\pi_1(G), SU(3))$ consisting of all homomorphisms
\begin{equation}
      \rho :\pi_1(G) \longrightarrow SU(3)
\end{equation}
such that for each meridian $m$, we have that the corresponding $x_m \in \pi_1(G)$ maps to an element $\rho (m)$ that is conjugate to $\Phi$.

\begin{lemma} \label{l5}
	Let $S,T \in SU(3)$ be conjugate to $\Phi$, then $ST$ is conjugate to $\Phi$ if and only if the 1-eigenspaces of $S$ and $T$ are orthogonal.
	
	At this time, the 1-eigenspaces of $ST$
 is orthogonal to that of $S$ and $T$.

\end{lemma}

\begin{proof}
	Suppose that $e_1, e_2$ are the 1-eigenvector of $S,T$ respectively, and $e_1 \bot e_2$. Choose $e_3$, so that $e_3 \bot e_1, e_2$. Then $Se_2 =-e_2$, $Se_3=-e_3$, $Te_1=-e_1$, $Te_3=-e_3$. Therefore $STe_1=-e_1, STe_2 =-e_2, STe_3=e_3$, which shows that $ST$ is conjugate to $\Phi$.
	
	Conversely, let $ST$ be conjugate to $\Phi$. Let $e_1,e_2,e_3$ be the eigenvectors of $T$, so that under the basis $\{e_1,e_2,e_3\}$, $S,T$ are represented respectively by the matrices
	\begin{equation}
		\begin{pmatrix}
			a_{11} & a_{12} & a_{13} \\
			a_{21} & a_{22} & a_{23} \\
			a_{31} & a_{32} & a_{33} \\
		\end{pmatrix}
	, \quad
	\begin{pmatrix}
		1 & 0 & 0 \\ 0 & -1 & 0 \\ 0 & 0 & -1 \\
	\end{pmatrix}
	\end{equation}
 	
 	Then
 	\begin{equation}
 		ST=\begin{pmatrix}
 			a_{11} &-a_{12}& -a_{13} \\
 			a_{21} &-a_{22}& -a_{23}\\
 			a_{31} &-a_{32}& -a_{33}\\
 		\end{pmatrix}
 	\end{equation}
 is conjugate to $\Phi$,
 $trS=a_{11}+a_{22}+a_{33}=-1$, $tr(ST)=a_{11}-a_{22}-a_{33}=-1$, which implies $a_{11}=-1, a_{22}+a_{33}=0$. Since $S \in SU(3)$, $|a_{11}|^2+|a_{12}|^2+|a_{13}|^2=1$, thus $a_{12}=a_{13}=0$.
 
 Let $xe_1+ye_2+ze_3$ be the eigenvector of $S$ with eigenvector 1. Then
 \begin{equation}
 	\begin{pmatrix}
 		-1 & 0 & 0 \\
 		a_{21} & a_{22} & a_{23} \\
 		a_{31} & a_{32} & a_{33}
 	\end{pmatrix}
 \begin{pmatrix}
 	x \\ y \\ z
 \end{pmatrix}
=
\begin{pmatrix}
	-x \\ \cdots \\ \cdots
\end{pmatrix}
=\begin{pmatrix}
	x \\ y \\ z
\end{pmatrix}
 \end{equation}
 therefore $x=0$, the 1-eigenspace of $S$ is orthogonal to the 1-eigenspace of $T$.
 
 The last argument follows since $T=S^2T=S(ST), S+ST^2=(ST)T$.
\end{proof}

For $\rho \in R_\Phi(\pi_1(G):SU(3))$, we will define a decoration $\mathcal{D}(\rho)$ of $G$ as follows. For each edge $e$ of $G$, there are two meridians along $e$, with reverse orientations. They induce reciprocal elements $x_e, -x_e$ of $\pi_1(G)$, therefore the images $\rho(x_e), \rho(-x_e)$  are inverse to each other. We define the evaluation of $\mathcal{D}(\rho)$ at $e$ to be the 1-eigenspace of $\rho(x_e)$, which does not depend on the choice of the orientation of the edge.

\begin{theorem} \label{thm2}
	The map described above 
	\begin{equation}
		\begin{aligned}
			R_\Phi(\pi_1(G);SU(3)) &\longrightarrow \mathcal{M}(G) \\
			\rho &\longrightarrow  \mathcal{D}(\rho)
		\end{aligned}
	\end{equation}
is a well-defined homeomorphism.
\end{theorem}

\begin{proof}
	  Let $e_1,e_2,e_3$ be adjacent to the same vertex of $G$. Choose any $\rho \in R_\Phi(\pi_1(G)\ ; SU(3))$. By the discussion above, we may assume that $\rho(e_2)\rho(e_1)=\rho(e_3)$. By Lemma \ref{l5}     $\mathcal{D}(\rho)(e_1),\mathcal{D}(\rho)(e_2),\mathcal{D}(\rho)(e_3)$ are orthogonal. Therefore $\mathcal{D}(\rho)$ is admissible.
	  
	  The map is bijective: We can define its inverse
	  \begin{equation}
	  	\begin{aligned}
	  	\mathcal{M}(G) &\longrightarrow	R_\Phi(\pi_1(G);SU(3)) \\
	  	\mathcal{D} &\longrightarrow \rho(\mathcal{D})
	  	\end{aligned}
	  \end{equation}
  such that  $\rho(\mathcal{D})$ is the unique matrix in $SU(3)$ that is conjugate to $\Phi$ and has $e$ as its 1-eigenvector. Obviously the two maps are continuous and are inverse to each other.
\end{proof}

In \cite{MR3880205}, they defined an invariant of trivalent spatial graphs $G$ which takes the form of a $\mathbb{Z}/2$-vector space $J^\#(G)$, using an variant of instanton homology. More precisely, it arises from a Chern-Simons functional whose set of critical points can be identified with the space
\begin{equation}
	\mathcal{R}(K)= \{ \rho : \pi_1(G) \longrightarrow SO(3)| \ \rho(x_e)\text{ has order 2 for each edge }  e\}
\end{equation}
Since a matrix $A \in SU(3)$ is conjugate to $\Phi$ if and only if it has order 2,  Theorem \ref{thm2}   implies that our moduli space $\mathcal{M}(G)$ is homeomorphic to the representation space
\begin{equation}
	\{ \rho : \pi_1(G) \longrightarrow SU(3)| \ \rho(x_e)\text{ has order 2 for each edge }  e\}
\end{equation}

In the definition of $\mathcal{M}(G)$, if we replace $\mathbb{CP}^2$ by $\mathbb{RP}^2$, then we get a moduli space $\widetilde{\mathcal{M}}(G)$, which is, by iterating the argument of this section, homeomorphic to 
$\mathcal{R}(K)$.

\textbf{Question}: What is the Euler characteristic of  $\widetilde{\mathcal{M}}(G)$?

	\bibliography{tcms}

\begin{thebibliography}{10}

\bibitem{2004math......2266G}
Bojan {Gornik}.
\newblock {Note on Khovanov link cohomology}.
\newblock {\em arXiv Mathematics e-prints}, page math/0402266, February 2004.

\bibitem{MR2100691}
Mikhail Khovanov.
\newblock sl(3) link homology.
\newblock {\em Algebr. Geom. Topol.}, 4:1045--1081, 2004.

\bibitem{MR4178907}
Mikhail Khovanov and Louis-Hadrien Robert.
\newblock Foam evaluation and {K}ronheimer-{M}rowka theories.
\newblock {\em Adv. Math.}, 376:Paper No. 107433, 59, 2021.

\bibitem{MR2391017}
Mikhail Khovanov and Lev Rozansky.
\newblock Matrix factorizations and link homology.
\newblock {\em Fund. Math.}, 199(1):1--91, 2008.

\bibitem{MR2860345}
P.~B. Kronheimer and T.~S. Mrowka.
\newblock Knot homology groups from instantons.
\newblock {\em J. Topol.}, 4(4):835--918, 2011.

\bibitem{MR3551837}
P.~B. Kronheimer and T.~S. Mrowka.
\newblock Exact triangles for {$SO(3)$} instanton homology of webs.
\newblock {\em J. Topol.}, 9(3):774--796, 2016.

\bibitem{MR3880205}
P.~B. Kronheimer and T.~S. Mrowka.
\newblock Tait colorings, and an instanton homology for webs and foams.
\newblock {\em J. Eur. Math. Soc. (JEMS)}, 21(1):55--119, 2019.

\bibitem{MR3966740}
Peter~B. Kronheimer and Tomasz~S. Mrowka.
\newblock Knots, three-manifolds and instantons.
\newblock In {\em Proceedings of the {I}nternational {C}ongress of
  {M}athematicians---{R}io de {J}aneiro 2018. {V}ol. {I}. {P}lenary lectures},
  pages 607--634. World Sci. Publ., Hackensack, NJ, 2018.

\bibitem{MR3956896}
Peter~B. Kronheimer and Tomasz~S. Mrowka.
\newblock A deformation of instanton homology for webs.
\newblock {\em Geom. Topol.}, 23(3):1491--1547, 2019.

\bibitem{MR1403861}
Greg Kuperberg.
\newblock Spiders for rank {$2$} {L}ie algebras.
\newblock {\em Comm. Math. Phys.}, 180(1):109--151, 1996.

\bibitem{MR3190356}
Andrew Lobb and Raphael Zentner.
\newblock The quantum {${\rm sl}(N)$} graph invariant and a moduli space.
\newblock {\em Int. Math. Res. Not. IMRN}, (7):1956--1972, 2014.

\bibitem{MR2443231}
Marco Mackaay and Pedro Vaz.
\newblock The foam and the matrix factorization {$\rm sl_3$} link homologies
  are equivalent.
\newblock {\em Algebr. Geom. Topol.}, 8(1):309--342, 2008.

\bibitem{MR1659228}
Hitoshi Murakami, Tomotada Ohtsuki, and Shuji Yamada.
\newblock Homfly polynomial via an invariant of colored plane graphs.
\newblock {\em Enseign. Math. (2)}, 44(3-4):325--360, 1998.

\end{thebibliography}
	\nocite{*}
\end{document}